\documentclass[a4paper,reqno,11pt]{amsart}

\usepackage{setspace, amssymb, amsmath, amsthm, graphicx, color, comment, booktabs, multirow, hyperref, mathtools, thmtools, thm-restate}
\usepackage[T1]{fontenc}
\usepackage[utf8]{inputenc}
\usepackage[inline]{enumitem}
\usepackage[dvipsnames,svgnames,table]{xcolor}
\usepackage[foot]{amsaddr}
\usepackage[capitalise, noabbrev]{cleveref}

\usepackage[longnamesfirst,numbers,sort&compress]{natbib}

\hypersetup{
    pdftitle={Strong coloring numbers of grids},
    pdfauthor={Piotr Micek, Torsten Ueckerdt},
    colorlinks,
    linkcolor={RoyalBlue},
    citecolor={RubineRed},
    urlcolor={blue!80!black},
    nesting=false,
}

\input{style.toc}

\colorlet{defcolor}{ForestGreen}
\newcommand{\defin}[1]{\relax\ifmmode{\color{defcolor}{#1}}\else{\emph{\textcolor{defcolor}{#1}}}\fi}

\DeclarePairedDelimiter\set{\{}{\}}

\newcommand{\calB}{\mathcal{B}}

\newcommand{\dist}{\mathrm{dist}}

\let\leq\leqslant
\let\geq\geqslant

\let\epsilon\varepsilon
\let\phi\varphi

\DeclareMathOperator\SReach{SReach}

\DeclareMathOperator\scol{scol}

\usepackage{tikz}
\usetikzlibrary{arrows}
\usetikzlibrary{calc}
\tikzset{myspacing/.style = {outer sep = 5pt, inner sep = 5pt}}
\tikzset{implies/.style = {-implies, double distance=3pt, thick}}
\tikzset{box/.style = {rectangle, draw, myspacing, align=center}}

\providecommand{\noopsort}[1]{}

\makeatletter
\def\thm@space@setup{
\thm@preskip=4mm
\thm@postskip=0mm
}
\makeatother

\crefformat{equation}{#2(#1)#3}
\let\eqref\cref
\crefformat{subsection}{Subsection #2#1#3}
\crefformat{subsubsection}{Subsubsection #2#1#3}

\theoremstyle{plain}
\newtheorem{thm}{Theorem}
\newtheorem*{thm*}{Theorem}

\newtheorem{lemma}[thm]{Lemma}
\newtheorem*{lemma*}{Lemma}

\newtheorem*{cor*}{Corollary}

\newtheorem*{conj*}{Conjecture}

\theoremstyle{remark}

\newtheorem*{problem*}{Open problem}

\newtheorem*{claim*}{Claim}

\crefname{obs}{Observation}{Observations}
\theoremstyle{definition}

\crefname{lem}{Lemma}{Lemmas}
\crefname{thm}{Theorem}{Theorems}
\crefname{cor}{Corollary}{Corollaries}

\setenumerate{label=\textup{(\roman*)}, noitemsep, topsep=3pt-\parskip, 
labelindent=.2em, leftmargin=*, widest=iii,}
\setitemize{noitemsep, topsep=-\parskip, labelindent=.2em, leftmargin=*, widest=iii,}

\newdimen\mywidth
\sbox0{m}%
\newdimen\mywidthprim
\sbox0{m}%
\newdimen\mywidthprimprim
\sbox0{m}%
\newdimen\mywidthA
\sbox0{m}%
\newdimen\mywidthAprim
\sbox0{m}%
\title{Strong coloring numbers of grids}

\begin{document}

\author[Micek]{Piotr Micek}
\address[P.~Micek]{Theoretical Computer Science Department,
Faculty of Mathematics and Computer Science, Jagiellonian University, Krak\'ow, Poland}
\email{piotr.micek@uj.edu.pl}

\author[Ueckerdt]{Torsten Ueckerdt}
\address[T.~Ueckerdt]{Institute of Theoretical Informatics,
Karlsruhe Institute of Technology, Karlsruhe, Germany}
\email{torsten.ueckerdt@kit.edu}

\thanks{
    Piotr Micek was funded by the National Science Centre, Poland under grant UMO-2023/05/Y/ST6/00079 within the WEAVE-UNISONO program.
    Torsten Ueckerdt was funded by the Deutsche Forschungsgemeinschaft (DFG, German Research Foundation) – 532213793.
}

\begin{abstract}
For positive integers $n$ and $d$, 
the $d$-dimensional grid of side length $n$ is the graph $G(n,d)$ with vertex set $[n]^d$, in which two vertices are adjacent if they differ by $1$ in exactly one coordinate and agree in all the other coordinates.
We determine, for every fixed $d$, the asymptotic growth in $r$ of the $r$th strong coloring numbers of these grids. 
More precisely for $n$ large enough, $\scol_r(G(n,d)) = \Omega_d\bigl(r^{d-1}\bigr)$ as $r\to\infty$.
This matches an upper bound given by a volume packing argument, which holds more generally for touching graphs of unit balls in $\mathbb R^d$.
For the lower bound, we construct in $G(n,d)$ a bramble of order $\Omega_d(n^{d-1})$ with each element of radius at most $dn$. 
This construction was found by OpenAI's GPT-5.6 Sol.
\end{abstract}

\maketitle

\section{Introduction}\label{sec:intro}

All graphs considered in this paper are finite, simple, and undirected.
Let $G$ be a graph, let $\sigma$ be a linear ordering of $V(G)$, and let $r$ be a nonnegative integer.
For vertices $u,v\in V(G)$, we say that $u$ is \defin{strongly $r$-reachable from $v$ in $(G,\sigma)$} if $u\leq_\sigma v$ and there is a path from $v$ to $u$ of length at most $r$ such that every internal vertex of the path is greater than $v$ in~$\sigma$.
Let \defin{$\SReach_r[G,\sigma,v]$} be the set of all vertices strongly $r$-reachable from $v$ in $(G,\sigma)$.
The \defin{$r$th strong coloring number} of $G$ is
\[
    \defin{\text{$\scol_r(G)$}}
    =\min_{\sigma}\max_{v\in V(G)}\bigl|\SReach_r[G,\sigma,v]\bigr|,
\]
where the minimum is taken over all linear orderings of $V(G)$.

Strong coloring numbers and their weak counterparts were introduced by Kierstead and Yang~\cite{KY03}. They provide a quantitative description of graph sparsity and, in particular, characterize classes of bounded expansion~\cite{Zhu09}: a graph class $\mathcal{C}$ has \defin{bounded expansion} if and only if there exists a function $f\colon\mathbb{N}\to\mathbb{N}$ such that
$\scol_r(G)\leq f(r)$ for every $G\in\mathcal{C}$ and every $r\in\mathbb N$. Examples include classes of bounded treewidth and all proper minor-closed classes, in particular the class of planar graphs. We refer to the book of Ne\v{s}et\v{r}il and Ossona de Mendez~\cite{sparsity} and the lecture notes of Pilipczuk, Pilipczuk, and Siebertz~\cite{notes} for further background.

For a positive integer $d$, let $\mathcal{C}_d$ denote the class of touching graphs of unit balls in $\mathbb{R}^d$; that is, intersection graphs of finite families of closed unit balls with pairwise disjoint interiors. 
Dvo\v{r}\'ak, Pek\'arek, Ueckerdt, and Yuditsky~\cite{DPUY22} asked for the asymptotic growth, as $r\to\infty$ with fixed $d$, of
$s_d(r):=\sup_{G\in\mathcal C_d}\scol_r(G)$. 
For every fixed $d\geq2$, the previously known bounds were
$\Omega_d\bigl(r^{d/2}\bigr)
\leq s_d(r)
\leq \mathcal O_d\bigl(r^{d-1}\bigr)$.
The upper bound follows by a standard volume packing argument; 
we include a short self-contained proof in \cref{sec:upper-bound}.

For positive integers $n$ and $d$, let \defin{$G(n,d)$} be the graph with vertex set $[n]^d$, in which distinct vertices $x=(x_1,\ldots,x_d)$ and $y=(y_1,\ldots,y_d)$ are adjacent if they differ by $1$ in exactly one coordinate and agree in all the other coordinates.
Thus, $G(n,d)$ is the Cartesian product  of $d$ copies of the $n$-vertex path, also known as the \emph{$d$-dimensional grid} of side length $n$.

Our main result is that for all positive integers $d$, $r$ and for large enough integer $n$, we have $\scol_r(G(n,d))=\Omega_d(r^{d-1})$. Since $d$-dimensional grids are in $\mathcal{C}_d$ we obtain $s_d(r)=\Theta_d(r^{d-1})$, so the upper bound given by the volume packing argument is asymptotically tight.

\begin{thm}
\label{thm:main}
Let $d$ and $r$ be positive integers, and let $q=\lfloor\frac{r}{4d}\rfloor+1$. 
Then for every $n\geq q$
\[
\scol_r(G(n,d)) \geq \scol_r(G(q,d)) \geq \left\lceil\frac{q^{d-1}}{2d}\right\rceil.
\]
Therefore, $\scol_r(G(n,d))=\Omega_d(r^{d-1})$.
\end{thm}



Within the proof, 
we construct a bramble in $G(n,d)$ of large order such that each element of the bramble has bounded radius. Recall that a bramble in a graph  $G$ is a family of connected vertex sets in $G$ any two of which touch, and its order is the minimum size of a vertex set meeting every member. Our bramble members actually pairwise intersect. We prove that if such a bramble has order at least $t$ and every member has radius at most $\rho$, then
$\scol_{4\rho}(G)\geq t$.
This is an (obvious) specialization of the relationship between shallow brambles and strong coloring numbers studied by Bousquet, Cames van Batenburg, Esperet, Joret, and Micek~\cite[Theorem 6]{BousquetEtAl25}.

The classical treewidth duality theorem by Seymour and Thomas~\cite{Seymour93} says that the maximum order of a bramble in $G$ is equal to the treewidth of $G$ plus one. 
Interestingly, the treewidth of $G(n,d)$ is known to be in $\Theta_d(n^{d-1})$. This follows from the bounds on carving width by Kozawa, Otachi, and Yamazaki~\cite{Kozawa2010}, see also a discussion by Marx and Sidiropoulos~\cite[Proposition 2.3]{Marx2014}.
Kozawa, Otachi and Yamazaki obtained the lower bound applying the edge-isoperimetric inequality derived by Bollobás and Leader~\cite{BL91}. 
As far as we know, we present the first explicit construction of a bramble of asymptotically maximum order in $d$-dimensional grids, for all $d\geq2$.

\section{Proof}\label{sec:proof}

For a positive integer $d$, we denote by \defin{$[d]$} the set $\{1,\dots,d\}$, and $\defin{[0]} = \emptyset$.

Let $G$ be a graph. 
Let $X$ be a subset of vertices of $G$. By $\defin{G[X]}$ we denote the subgraph of $G$ induced by $X$. 
We say that $X$ is \defin{connected} if $G[X]$ is a connected graph.
The \defin{distance} between two vertices $u$ and $v$ in $G$, denoted by $\defin{\dist_G(u,v)}$, is the minimum length (number of edges) of a $u$--$v$ path in $G$. 
The \defin{radius} of $G$ is the minimum integer $r$ such that there exists a vertex $v\in V(G)$ with  
$\dist_G(v,u)\leq r$ for all $u\in V(G)$.

An \defin{intersecting bramble} in $G$ is a family $\calB$ of nonempty connected subsets of $V(G)$ such that any two members of $\calB$ intersect.
The \defin{order} of $\calB$ is the minimum size of a set $X\subseteq V(G)$ that intersects every member of $\calB$.
We say that $\calB$ has \defin{radius} at most $\rho$ if the graph $G[B]$ has radius at most $\rho$ for every $B\in\calB$.

The next lemma connects brambles of bounded radius with a lower bound for a strong coloring number.
The proof is almost one-to-one in~\cite{BousquetEtAl25}. 
We include it for completeness.

\begin{lemma}[\cite{BousquetEtAl25}]\label{lem:bramble_to_scol}
    Let $G$ be a graph, and let $\calB$ be an intersecting bramble in $G$ of radius at most $\rho$ and order at least $t$.
    Then
    \[
        \scol_{4\rho}(G)\geq t.
    \]
\end{lemma}

\begin{proof}
    Fix a linear ordering $\sigma$ of $V(G)$.
    For each $B\in\calB$, let $m(B)$ be the minimum vertex of $B$ in $\sigma$.
    Choose $B_0\in\calB$ so that $v=m(B_0)$ is maximum in $\sigma$ among the vertices in $\{m(B)\mid B\in\calB\}$.
    We claim that $\SReach_{4\rho}[G,\sigma,v]$ intersects every member of~$\calB$. 

    Fix $B\in\calB$. We aim to prove that 
    $B\cap \SReach_{4\rho}[G,\sigma,v] \neq\emptyset$. 
    If $v\in B$, then there is nothing to prove as $v\in\SReach_{4\rho}[G,\sigma,v]$.
    Suppose that $v\notin B$. 
    Since $\calB$ is an intersecting bramble, In this case, $B_0\neq B$ and $m(B)<_\sigma m(B_0)=v$. 
    Since $\calB$ is an intersecting bramble $B_0\cap B\neq\emptyset$ and we fix a vertex $x\in B_0\cap B$.
    Since $G[B_0]$ and $G[B]$ both have radius at most $\rho$, there is a path $P_0$ in $G[B_0]$ from $v$ to $x$ of length at most $2\rho$, and there is a path $P$ in $G[B]$ from $x$ to $m(B)$ of length at most $2\rho$.
    Every vertex of $P_0 - \{v\}$ is greater than $v$ in $\sigma$.
    In particular $v <_{\sigma} x$. 
    Traverse the second path $P$ from $x$ towards $m(B)$, and let $u$ be the first vertex on $P$ smaller than $v$ in $\sigma$. 
    It exists as $m(B)<_\sigma v$. 
    Every vertex encountered before $u$ is greater than $v$.
    Concatenating $P_0$ and the subpath of $P$ from $x$ to $u$ gives a subgraph of $G$ that contains a path from $v$ to $u$ of length at most $4\rho$ whose internal vertices are greater than $v$.
    Hence, $u\in\SReach_{4\rho}[G,\sigma,v]\cap B$. This completes the proof of the claim.

    Thus, $\SReach_{4\rho}[G,\sigma,v]$ intersects every element of $\calB$, and consequently it has size at least $t$.
    Since $\sigma$ was arbitrary, $\scol_{4\rho}(G)\geq t$.
\end{proof}

We now construct a suitable intersecting bramble in a $d$-dimensional grid.

\begin{lemma}\label{lem:grid_bramble}
    Let $q$ and $d$ be positive integers.
    The grid $G(q,d)$ contains an intersecting bramble of radius at most $d(q-1)$ and order at least
    \[
        \left\lceil\frac{q^{d-1}}{2d}\right\rceil.
    \]
\end{lemma}

\begin{proof}
    Write $G=G(q,d)$.
    For $x,y\in[q]^d$, let $P(x,y)$ be the shortest path from $x$ to $y$ in $G$ that changes the coordinates of $x$ to the corresponding coordinates of $y$, successively in the order $1,\ldots,d$.
    We call $P(x,y)$ the canonical path from $x$ to $y$. 
    Note that the length of $P(x,y)$ is at most
    $d(q-1)$. 

    Consider the family $\set{P(x,y)\mid x,y\in[q]^d}$ of all canonical paths in $G$.
    We first bound the congestion of this family. 
    We claim that every $z\in[q]^d$ belongs to at most
    \begin{equation}\label{eq:congestion}
        d q^{d+1}
    \end{equation}
    canonical paths.
    To this end, suppose that $z$ lies in $P(x,y)$. 
    By the definition of the canonical path, 
    there exists $i\in[d]$ such that 
    \[
        z_j=x_j\quad\text{for every $j<i$}
        \qquad
        \text{and}
        \qquad
        z_j=y_j\quad\text{for every $j>i$}.
    \]
    Now for a fixed value of $i$, 
    there are at most $q^{i-1}$ choices for $(x_1,\ldots,x_{i-1})$, at most $q^{d-i}$ choices for $(y_{i+1},\ldots,y_d)$, and at most $q^2$ choices for $(x_i,y_i)$. 
    This gives at most $q^{i-1}\cdot q^{d-i}\cdot q^2=q^{d+1}$ choices of $(x,y)$. 
    Hence, the number of canonical paths $P(x,y)$ containing $z$ is upper bounded by $d\cdot q^{d+1}$, which proves~\eqref{eq:congestion}.

    Let $S\subseteq[q]^d$ be any fixed subset of vertices satisfying
    \begin{equation}\label{eq:S_small}
        |S|<\frac{q^{d-1}}{2d}.
    \end{equation}
    By~\eqref{eq:congestion}, the number of pairs $(x,y)$ for which $P(x,y)$ intersects $S$ is at most
    \[
        |S|\,d q^{d+1} < \frac{q^{d-1}}{2d} \cdot d q^{d+1} = \frac{1}{2}q^{2d}.
    \]
    Averaging over the $q^d$ possible choices of $y$, we can choose an anchor $y_S$ such that
    \[
        \big|\{x\in[q]^d\mid V(P(x,y_S))\cap S=\emptyset\}\big|>\frac{1}{2}q^d.
    \]
    Let $B_S=\set{x\in[q]^d\mid V(P(x,y_S))\cap S=\emptyset}$. 
    Note that every vertex of $B_S$ can be joined to $y_S$ inside $G[B_S]$ by a canonical path. Altogether we have (1) $B_S\cap S=\emptyset$, 
    (2) $|B_S|>q^d/2$, (3) $G[B_S]$ is connected, and 
    (4) the radius of $G[B_S]$ is at most $d(q-1)$. 

    Let
    \[
        \calB=\left\{B_S\ \middle|\ S\subseteq[q]^d,\ |S|<\frac{q^{d-1}}{2d}\right\}.
    \]
    Any two members of $\calB$ intersect, since each contains more than half of the vertices of $G$.
    Furthermore, no set $X$ with $|X|<q^{d-1}/(2d)$ is a transversal of $\calB$, because $B_X \in \calB$ and $B_X\cap X=\emptyset$.
    Thus, $\calB$ is an intersecting bramble of order at least $\lceil q^{d-1}/(2d)\rceil$ and radius at most $d(q-1)$.
\end{proof}

\begin{proof}[Proof of \Cref{thm:main}]

    Let $q=\lfloor r/(4d)\rfloor+1$ and $n \geq q$.
    Then $r\geq 4d(q-1)$ and $G(q,d)$ is a subgraph of $G(n,d)$.
    Together this gives
    \[
        \scol_r(G(n,d)) 
        \geq \scol_r(G(q,d))
        \geq 
        \scol_{4d(q-1)}(G(q,d))
        \geq \left\lceil\frac{q^{d-1}}{2d}\right\rceil,
    \]
    where the last inequality follows by \cref{lem:grid_bramble,lem:bramble_to_scol}.
\end{proof}

\section{Upper bound for touching graphs of unit balls}
\label{sec:upper-bound}

Our main result (\cref{thm:main}) states that for positive integers $d$ and $r$ the $r$th strong coloring number of the $d$-dimensional grid satisfies $\scol_r(G(n,d)) \geq \frac2{(4d)^d} \cdot r^{d-1} = \Omega_d(r^{d-1})$ if $n$ is large enough.
In this section we give an asymptotically matching upper bound by considering the larger class $\mathcal{C}_d$ of all touching graphs $G$ of unit balls in $\mathbb{R}^d$ with pairwise disjoint interiors.
Equivalently, each vertex of $G$ is a point in $\mathbb{R}^d$, any two vertices have Euclidean distance at least~$2$, and two vertices form an edge if and only if their distance is exactly~$2$.
In order to give an upper bound on the strong coloring number, we take the order $\sigma$ of vertices of $G$ by their first coordinate and argue that for every vertex $v$ and every integer $r$, the set $\SReach_r[G,\sigma,v]$ is small, since every vertex $u$ in $\SReach_r[G,\sigma,v]$ is close to $v$ in $\mathbb{R}^d$;
in particular all of $\SReach_r[G,\sigma,v]$ lies within a cylinder of height~$4$ and radius $2r+1$.
Dividing the volume of the cylinder by the volume of a unit ball then gives the desired upper bound.

More formally, for $d\geq 1$, let $\omega_d = \frac{2(2\pi)^{(d-1)/2}}{d!!}$ denote the $d$-dimensional volume of the Euclidean unit ball in $\mathbb{R}^d$, and put $\omega_0=1$.
Then the following holds.

\begin{thm}
  Let $d\geq 1$, and let $G$ be the touching graph of a finite family of closed unit balls in $\mathbb{R}^d$ whose interiors are pairwise disjoint.
  Then, for every $r \geq 1$, we have
  \[
    \scol_r(G)
    \leq
    \frac{4\omega_{d-1}}{\omega_d}(2r+1)^{d-1}.
  \]
  In particular, $\scol_r(G)=\mathcal{O}_d(r^{d-1})$.
\end{thm}
\begin{proof}
 By taking the center of each unit ball, each vertex $v\in V(G)$ is a point $v = (v_1,\ldots,v_d) \in \mathbb{R}^d$, so that $\lVert u-v \rVert \geq 2$ for every $u \neq v$ and $\lVert u-v \rVert = 2$ if and only if $uv \in E(G)$.
 We may assume without loss of generality that the first coordinates of all points in $V(G)$ are pairwise distinct.
 Let $\sigma$ be the ordering of $V(G)$ by increasing values of their first coordinate in $\mathbb{R}^d$, i.e., $u <_\sigma v$ if and only if $u_1 < v_1$.
  
 Now take a fixed vertex $v\in V(G)$ and consider any vertex $u\in\SReach_r[G,\sigma,v]$.
 Let
 \[
  P=(x^0,x^1,\ldots,x^\ell),
    \qquad x^0=v,\quad x^\ell=u,\quad \ell\leq r,
 \]
 be a path witnessing that $u \in \SReach_r[G,\sigma,v]$.
 Then, in particular $u <_\sigma v$ and $v \leq_\sigma x^{\ell-1}$, which together with $\lVert u-x^\ell \rVert = 2$ gives
 \begin{equation}
    v_1 - 2 \leq u_1 < v_1. \label{eq:1}
 \end{equation}
 Secondly, as $P = (x^0,\ldots,x^\ell)$ is a path from $v$ to $u$ of length $\ell \leq r$, we have 
 \begin{equation}
     \lVert u-v \rVert \leq \sum_{i=0}^{\ell-1} \lVert x^i - x^{i+1} \rVert = 2\ell \leq 2r. \label{eq:2}
 \end{equation}

 \begin{figure}
     \centering
     \includegraphics{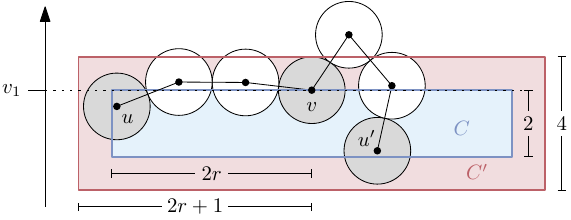}
     \caption{$v,u,u'\in \SReach_{r}(G,\sigma,v)$. The three discs have pairwise disjoint interiors and they are all contained in a cylinder $C'$ of height $4$ and radius $2r+1$.}
     \label{fig:scol-cylinders}
 \end{figure}

 Together, \eqref{eq:1,eq:2} give that the point $u$ lies in a cylinder $C$ of height~$2$ and radius~$2r$ with point $v$ as the center of one of the spherical facets of $C$.
 Consequently, the unit ball with center $u$ lies entirely inside the cylinder $C'$ of height~$4$ and radius~$2r+1$ consisting of all points of distance at most~$1$ to $C$.
 See \cref{fig:scol-cylinders} for an illustrating example.
 Cylinder $C' = B \times I$ is the Cartesian product of a $(d-1)$-dimensional ball $B$ of radius $2r+1$ and a $1$-dimensional interval $I$ of length~$4$.
 Hence, the $d$-dimensional volume of $C'$ is given by
 \[
    \operatorname{vol}_d(C') =4\omega_{d-1}(2r+1)^{d-1}.
 \]
 The unit balls corresponding to the vertices in $\SReach_r[G,\sigma,v]$ have pairwise disjoint interiors, each has volume $\omega_d$, and all of them are completely contained in $C'$.
 Hence, it follows that
 \[
    \bigl|\SReach_r[G,\sigma,v]\bigr|\,\omega_d \leq 4\omega_{d-1}(2r+1)^{d-1},
 \]
 which concludes the proof.
\end{proof}

\section*{Statement of AI use}
The construction in \cref{lem:grid_bramble} was found by OpenAI’s GPT-5.6 Sol. 
The authors take responsibility for the mathematical correctness of the presented arguments.

\bibliographystyle{plain}
\bibliography{biblio}

\begin{thebibliography}{10}

\bibitem{BL91}
Béla Bollobás and Imre Leader.
\newblock Edge-isoperimetric inequalities in the grid.
\newblock {\em Combinatorica}, 11:299--314, 1991.

\bibitem{BousquetEtAl25}
Nicolas Bousquet, Wouter~Cames van Batenburg, Louis Esperet, Gwena{\"e}l Joret,
  and Piotr Micek.
\newblock Shallow brambles.
\newblock {\em Discrete Mathematics \& Theoretical Computer Science}, 27(3),
  2025.
\newblock \href{https://arxiv.org/abs/2502.04177}{arXiv:2502.04177}.

\bibitem{DPUY22}
Zden\v{e}k Dvo\v{r}\'ak, Jakub Pek\'arek, Torsten Ueckerdt, and Yelena
  Yuditsky.
\newblock Weak coloring numbers of intersection graphs.
\newblock In {\em 38th International Symposium on Computational Geometry (SoCG
  2022)}, volume 224 of {\em Leibniz International Proceedings in Informatics
  (LIPIcs)}, pages 39:1--39:15. Schloss Dagstuhl -- Leibniz-Zentrum f{\"u}r
  Informatik, 2022.
\newblock \href{https://arxiv.org/abs/2103.17094}{arXiv:2103.17094}.

\bibitem{KY03}
Hal~A. Kierstead and Daqing Yang.
\newblock Orderings on graphs and game coloring number.
\newblock {\em Order}, 20(3):255--264, 2003.

\bibitem{Kozawa2010}
Kyohei Kozawa, Yota Otachi, and Koichi Yamazaki.
\newblock The carving-width of generalized hypercubes.
\newblock {\em Discrete Mathematics}, 310(21):2867--2876, 2010.

\bibitem{Marx2014}
D{\'a}niel Marx and Anastasios Sidiropoulos.
\newblock The limited blessing of low dimensionality: when 1-1/d is the best
  possible exponent for d-dimensional geometric problems.
\newblock In {\em Proceedings of the Thirtieth Annual Symposium on
  Computational Geometry}, SOCG'14, page 67–76, New York, NY, USA, 2014.
  Association for Computing Machinery.
\newblock \href{https://arxiv.org/abs/1612.01171}{arXiv:1612.01171}.

\bibitem{sparsity}
Jaroslav Ne\v{s}et\v{r}il and Patrice {Ossona de Mendez}.
\newblock {\em Sparsity --- {G}raphs, {S}tructures, and {A}lgorithms},
  volume~28 of {\em Algorithms and {C}ombinatorics}.
\newblock Springer, 2012.

\bibitem{notes}
Marcin Pilipczuk, Micha\l{} Pilipczuk, and Sebastian Siebertz.
\newblock Lecture notes for the course ``{S}parsity'' given at {F}aculty of
  {M}athematics, {I}nformatics, and {M}echanics of the {U}niversity of
  {W}arsaw, winter semesters 2017/18 and 2019/20.
\newblock \url{https://www.mimuw.edu.pl/~mp248287/sparsity2}.

\bibitem{Seymour93}
Paul~D. Seymour and Robin Thomas.
\newblock Graph searching and a min-max theorem for tree-width.
\newblock {\em Journal of Combinatorial Theory, Series B}, 58(1):22--33, 1993.

\bibitem{Zhu09}
Xuding Zhu.
\newblock Colouring graphs with bounded generalized colouring number.
\newblock {\em Discrete Mathematics}, 309(18):5562--5568, 2009.

\end{thebibliography}

\end{document}